\documentclass[11pt,reqno]{amsart}
\usepackage{amssymb,amsmath}
\usepackage[utf8]{inputenc}
\usepackage{amsthm}
\usepackage{amssymb}
\usepackage{amsmath,amscd}
\usepackage{latexsym}
\usepackage{graphicx,subfigure}
\usepackage[
    hscale=0.70,
    vscale=0.775,
    headheight=13pt,
    centering,
]{geometry}
\usepackage{tikz-cd}
\usepackage{xcolor}
\definecolor{verydarkblue}{rgb}{0,0,0.5}
\usepackage[
    breaklinks,
    colorlinks,
    citecolor=verydarkblue,
    linkcolor=verydarkblue,
    urlcolor=verydarkblue,
    pagebackref=true,
    hyperindex
]{hyperref}

\usepackage[capitalize]{cleveref}

\theoremstyle{plain}

\newtheorem{introtheorem}{Theorem}

\crefname{introtheorem}{Theorem}{Theorems}

\newtheorem{theorem}{Theorem}
\newtheorem{lemma}[theorem]{Lemma}
\newtheorem{proposition}[theorem]{Proposition}

\newtheorem{corollary}[theorem]{Corollary}

\newtheorem{question}[theorem]{Question}

\theoremstyle{definition}
\newtheorem{definition}[theorem]{Definition}
\newtheorem{example}[theorem]{Example}

\theoremstyle{remark}

\newtheorem{remark}{Remark}

\numberwithin{theorem}{section}
\numberwithin{equation}{section}
\numberwithin{remark}{section}

\newcommand{\Q}{\mathbb{Q}}
\newcommand{\PP}{\mathbb{P}}

\newcommand{\C}{\mathbb{C}}

\newcommand{\aid}{\mathfrak{a}}

\newcommand{\fm}{\mathfrak{m}}

\newcommand{\OS}{\mathcal{O}}

\newcommand{\HH}{\mathcal{H}}

\newcommand{\spec}{\operatorname{Spec}}

\newcommand{\codim}{\operatorname{codim}}

\newcommand{\ord}{\operatorname{ord}}
\newcommand{\vol}{\operatorname{vol}}

\newcommand{\lct}{\operatorname{lct}}
\newcommand{\ci}{\operatorname{ci}}

\newcommand{\bs}{\operatorname{Bs}}

\newcommand{\bi}[2]{\left(\begin{array}{c}
#1\\
#2
\end{array}\right)}

\title[Higher Codimensional Alpha Invariants]{Higher Codimensional Alpha Invariants and Characterization of Projective Spaces}
\author{Ziwen Zhu}
\address{Department of Mathematics, University of Utah, Salt Lake City, UT 84112, USA}
\email{{\tt zzhu@math.utah.edu}}
\thanks{Research partially supported by NSF Grant DMS-1402907}

\begin{document}
\begin{abstract}
We generalize the definition of alpha invariant to arbitrary codimension. We also give a lower bound of these alpha invariants for K-semistable $\Q$-Fano varieties and show that we can characterize projective spaces among all K-semistable Fano manifolds in terms of higher codimensional alpha invariants. Our results demonstrate the relation between alpha invariants of any codimension and volumes of Fano manifolds in the characterization of projective spaces.
\end{abstract}
\maketitle
\section{Introduction}
We work over the complex number field $\C$. A variety $X$ is called $\Q$-Fano if $X$ is a normal projective variety with klt singularities such that the anti-canonical divisor $-K_X$ is an ample $\Q$-divisor. A Fano manifold is a smooth $\Q$-Fano variety. we will always use $n$ to denote the dimension of the $\Q$-Fano variety $X$.

It is well-known that a Fano manifold admits a K\"ahler-Einstein metric if and only if it is K-polystable due to \cite{D1,D2,D3,tian}. More generally, we would like to study K-semistable $\Q$-Fano varieties. Recent work of Kento Fujita, Yuji Odaka and Chen Jiang shows that among K-semistable Fano manifolds, the projective space $\PP^n$ can be characterized by either of the following two properties:
\begin{enumerate}
\item \cite{fujita2015optimal} $(-K_X)^n\geq (n+1)^n$;
\item \cite{fujita2016k,sma} $\displaystyle{\alpha(X)\leq \frac{1}{n+1}}$.
\end{enumerate}
Here $(-K_X)^n$ is the volume of $X$, and $\alpha(X)$ is the alpha invariant of $X$.

The purpose of this paper is to show that the above two characterizations of projective spaces are special cases of a more general one where cycles of intermediate codimensions are considered. 

We first generalize the definition of alpha invariant:
\begin{definition}
Let $X$ be a $\Q$-Fano variety of dimension $n$. For $1\leq k\leq n$, the complete intersection $\ci(L_1,\ldots,L_k)$ in $X$ cut out by effective Cartier divisors $L_1,\ldots,L_k$ is defined to be the scheme-theoretic intersection of $L_1,\ldots,L_k$ with the expected codimension $k$. Then we define the alpha invariant of codimension $k$ for $X$ to be
$$
\alpha^{(k)}(X):=\inf_{r}\left\{\lct\left(X,\frac{1}{r}Z\right)\bigg|Z=\ci(L_1,\ldots,L_k),~L_1,\ldots,L_k\in |-rK_X|\right\}.
$$
\end{definition}
\begin{remark}
When $k=1$, the generalized alpha invariant $\alpha^{(1)}(X)$ is just the usual alpha invariant $\alpha(X)$. We will use $\alpha^{(1)}(X)$ to denote the usual alpha invariant for the rest part of the paper.
\end{remark}
Tian proved in \cite{tian_alpha} that a Fano manifold $X$ of dimension $n$ admits a K\"ahler-Einstein metric if $\alpha^{(1)}(X)>n/(n+1)$. Fujita improved the theorem in \cite{fujita_2017} by showing that a Fano manifold $X$ of dimension $n$ is K-stable if $\alpha^{(1)}(X)\geq n/(n+1)$. A recent related result by Stibitz and Zhuang in \cite{stibitz2018k} shows that a birationally superrigid (or more generally log maximal singularity free) Fano variety $X$ is K-stable (resp. K-semistable) if $\alpha^{(1)}(X)>1/2$ (resp. $\alpha^{(1)}(X)\geq1/2$). The result is later improved by Zhuang in \cite{zhuang2018birational}. As a first application of higher codimensional alpha invariants, we note that their results can be stated in terms of the higher codimensional alpha invariant $\alpha^{(2)}$. (See Theorem \ref{fsz} in Section \ref{min}).


For K-semistable $\Q$-Fano varieties, we can give a lower bound of higher codimensional alpha invariants as our first main result:
\begin{introtheorem}\label{ineq}
Let $X$ be a K-semistable $\Q$-Fano variety of dimension $n$. Then 
\begin{equation}\label{inf}
\alpha^{(k)}(X)\geq \frac{k}{n+1}.
\end{equation}
\end{introtheorem}
Note that when $k=1$, the inequality \eqref{inf} is proved by Fujita and Odaka in \cite{fujita2016k}.

It is well-known that $\PP^n$ is K-semistable. Therefore by considering the log canonical thresholds of linear subspaces of $\PP^n$, together with Theorem \ref{ineq}, we know that the equality holds in \eqref{inf} when $X\cong \PP^n$. Then we have our second main result about characterization of projective spaces:
\begin{introtheorem}\label{eq}
Let $X$ be a K-semistable Fano manifold of dimension $n$.
Consider the following three statements:
\begin{enumerate}
\item $X\cong \PP^n$;
\item $\displaystyle{\alpha^{(k)}(X)=\frac{k}{n+1}}$, and it is realized by some complete intersection $Z$;
\item $(-K_X)^k$ is rationally equivalent to $lZ'$ for some rational number $l\geq (n+1)^k$ and $Z'$ an integral $(n-k)$-cycle.
\end{enumerate}
We have $(1) \Rightarrow (2)\Rightarrow(3)$. Moreover, if we assume that $k$ divides $n$, then $(3) \Rightarrow (1)$ and therefore all three statements are equivalent. 
\end{introtheorem}

\begin{remark}
Note that in the second statement of Theorem \ref{eq}, we say $\alpha^{(k)}(X)$ is realized by a complete intersection $Z$ if $Z$ is the complete intersection of $L_1\ldots,L_k\in |-rK_X|$, such that $\lct(X,\frac{1}{r}Z)=\alpha^{(k)}(X)$. In this case, the infimum in the definition of $\alpha^{(k)}(X)$ is in fact a minimum. When $\alpha^{(1)}(X)\leq 1$, Birkar shows in \cite{birkar2016singularities} that $\alpha^{(1)}(X)$ is realized by some $D\in |-K_X|_\Q$. In general for higher codimensional alpha invariants, it is not clear whether they can always be realized by some complete intersection. 
\end{remark}

When $k=1$, Theorem \ref{eq} reduces to the main result in \cite{sma} that characterizes projective spaces among all K-semistable Fano manifolds in terms of the usual alpha invariant. In \cite{sma}, Jiang proves that (2) implies (3) first. $(3) \Rightarrow (1)$ follows from the main result of \cite{Kobe}. Also note that in this case, due to Birkar's result, we do not need to assume $\alpha^{(1)}(X)$ is realized in the second statement of Theorem \ref{eq}.

When $k=n$, using the inequality in \cite{DFEM}, a quick proof of Theorem \ref{eq} is given in Section \ref{min} (See Corollary \ref{casen}). In this case, it it also not necessary to assume $\alpha^{(n)}(X)$ is realized. Note that $(3)\Rightarrow (1)$ when $k=n$ is implied by the following theorem:
\begin{theorem}[\cite{fujita2015optimal}]\label{fu}
Let $X$ be a K-semistable $\Q$-Fano variety of dimension $n$. Then we have $(-K_X)^n\leq (n+1)^n$. Moreover if $X$ is smooth and $(-K_X)^n= (n+1)^n$, then we know that $X\cong \PP^n$.
\end{theorem}
In fact we use Theorem \ref{fu} to prove the last part of Theorem \ref{eq} for any $k$ that divides $n$.
We would also like to comment that in both \cite{liu2016volume} and \cite{Liu2017}, Liu and Zhuang proved a stronger version of Theorem~\ref{fu} without assuming smoothness.

\subsection*{Acknowledgements}
The author would like to thank his advisor Tommaso de Fernex for guiding his research and providing insightful thoughts throughout the project. He would also like to thank Harold Blum, Yuchen Liu, Ziquan Zhuang and Ivan Cheltsov for effective discussions. 
\section{Preliminaries}

\subsection{Cycles, rational equivalence and numerical equivalence}
Let $X$ be a scheme and $Z$ a $k$-dimensional subscheme of $X$. Let $Z_1,\ldots,Z_t$ be the irreducible components of $Z$. Then following the notation of \cite{fulton2012intersection}, the k-cycle of a subscheme Z is $[Z]=\sum_{i=1}^t a_i[Z_i]$, where $a_i=l(\OS_{Z,Z_i})$ is the length of $\OS_{Z,Z_i}$ as an $\OS_{Z,Z_i}$-module.

For any two $k$-cycles $Z$ and $W$ on a proper scheme, we write $Z\sim_{rat} W$ if $Z$ is rationally equivalent to $W$; and write $Z\equiv_{num} W$ if $Z$ is numerically equivalent to $W$. We refer to \cite{fulton2012intersection} for definitions. 

In particular, we know that rational equivalence is the same as linear equivalence for divisors, and we write $D\sim_{lin}D'$ if $D$ is linear equivalent to $D'$. Also on any smooth proper scheme $X$ of dimension $n$, two $k$-cycles $Z$ and $W$ are numerically equivalent if and only if we have the equality of intersection numbers 
$
Z\cdot T=W\cdot T
$
for any $(n-k)$-cycle $T$.
\subsection{Multiplicity of ideals}
Let $X$ be a scheme of dimension $n$, and $Z\subset X$ a closed subscheme corresponding to the ideal sheaf $I_Z\subset \OS_X$. Let $V$ be an irreducible component of $Z$ with codimension $k$. Then the multiplicity of $X$ along $Z$ at the generic point of $V$ is defined as
$$
e(I_Z\cdot\OS_{X,V}):=\lim_{t\to \infty}\frac{l(\OS_{X,V}/I_Z^t\cdot\OS_{X,V})}{t^k/k!}.
$$
Suppose in addition $Z$ is irreducible. Let $\sigma:Y\to X$ be a proper birational morphism such that $\sigma^{-1}I_Z\cdot \OS_Y=\OS_Y(-E)$. Then we have the following equality of $(n-k)$-cycles:
\begin{equation}\label{ram}
e(I_Z\cdot\OS_{X,V})[V]=(-1)^{k-1}\sigma_\ast (E^{k}).
\end{equation}
Refer to \cite{fulton2012intersection} for more details about multiplicity.

For the multiplicity of an ideal corresponding to a complete intersection, we have the following property (See Example 4.3.5 of \cite{fulton2012intersection} for a proof):
\begin{proposition}\label{cim}
Let $(R,\fm)$ be a Noetherian local ring of dimension $k$ and $\aid=(x_1,\ldots,x_k)$ an $\fm$-primary ideal. If $R$ is Cohen-Macaulay, then $e(\aid)=l(R/\aid)$.
\end{proposition}
\subsection{Singularities and log canonical threshold}
Let $X$ be a normal variety such that $K_X$ is $\Q$-Cartier and $Z\subset X$ a closed subscheme. Denote by $I_Z$ the ideal sheaf defining $Z$. We say $E$ is a divisor over $X$ if $E$ is a prime divisor on some normal variety $Y$ with a proper birational morphism $\sigma:Y\to X$, and $\sigma(E)$ is called the center of $E$ on $X$. We say that $E$ is exceptional if the center of $E$ on $X$ has dimension smaller than $E$. Use $A_X(E)$ to denote the log discrepancy of $E$. For $a>0$, we say that the pair $(X,aZ)$ has klt (or log canonical) singularities if $A_X(E)-a\ord_E(I_Z)> 0$ (or $\geq 0$) for any prime divisor $E$ over $X$. Let $\HH$ be a linear system on $X$ and $Z$ be the base scheme $\bs\HH$. Singularities of the pair $(X,a\HH)$ are defined to be the same as the singularities of the pair $(X,aZ)$. In the definition of klt and log canonical singularities, it is enough to examine finitely many divisors on a fixed log resolution. We refer to \cite{kollar2008birational} and \cite{lazarsfeld2004positivity2} for more details.

Let $X$ have klt singularities and assume $Z$ is non-empty. For any nonnegative number $a$, the log canonical threshold of the formal power $I^a_Z$ is defined to be the log canonical threshold of the pair $(X,aZ)$, which is a nonnegative number computed by 
\begin{equation}\label{lct}
\lct(I^a_Z):=\lct(X,aZ)=\inf_E\frac{A_X(E)}{a\ord_E(I_Z)},
\end{equation}
where the infimum runs through all prime divisors $E$ over $X$. Note that immediately from the definition, the log canonical threshold of $(X,aZ)$ can also be viewed as the largest number $t>0$ such that $(X,taZ)$ is log canonical. Then by fixing a log resolution of $(X,Z)$, we see that the infimum in \eqref{lct} is in fact a minimum running through finitely many divisors. In particular, log canonical thresholds of pairs are rational numbers. Also from the definition, for $a>0$, we have~$\lct(X,aZ):=\frac{1}{a}\lct(X,Z)$.   
\subsection{K-stability}
The $\beta$-invariant of a divisor over a $\Q$-Fano variety is defined in the following way:
\begin{definition}[\cite{fujita2016valuative}]
Let $X$ be a $\Q$-Fano variety of dimension $n$, and $F$ a prime divisor over $X$. Suppose $\sigma:Y\to X$ is a projective birational morphism such that $F$ is a divisor on $Y$. Then we define the $\beta$-invariant of $F$ to be
$$
\beta(F):=A_X(F)\vol_X(-K_X)-\int_0^\infty \vol_Y(\sigma^\ast(-K_X)-xF)\,dx.
$$
\end{definition}

Instead of recalling the original definition of K-stability via normal test configurations and Donaldson-Futaki invariants, we use a valuative criterion of K-semistability in terms of $\beta$-invariants developed by Fujita and Li.
\begin{theorem}[\cite{fujita2016valuative,li2015k}]\label{beta}
Let $X$ be a $\Q$-Fano variety. Then $X$ is K-semistable if and only if $\beta(F)\geq 0$ for any divisor $F$ over $X$.
\end{theorem}
We can also define $\beta$-invariants for proper closed subschemes of $X$ in a similar way.
\begin{definition}[\cite{fujita2015optimal}]
Let $X$ be a $\Q$-Fano variety of dimension $n$ and $Z\subset X$ a subscheme of $X$ defined by the ideal sheaf $I_Z\subset \OS_X$. Take a projective birational morphism $\sigma:Y\to X$ that factors through the blow-up of $X$ along $Z$, and write $\sigma^{-1}I_Z\cdot \OS_Y=\OS_Y(-F)$. Then we define the $\beta$-invariant of $Z$ to be
$$
\beta(Z):=\lct(X,Z)\vol_X(-K_X)-\int_0^\infty \vol_Y(\sigma^\ast(-K_X)-xF)\,dx.
$$
\end{definition}
Note that $\beta(F)$ and $\beta(Z)$ do not depend on the choice of the birational morphism~$\sigma$. An immediate consequence of Theorem \ref{beta} is the following result.
\begin{theorem}[\cite{fujita2015optimal}]\label{betaZ}
Let $X$ be a $\Q$-Fano variety. If $X$ is K-semistable, then $\beta(Z)\geq 0$ for any proper closed subscheme $Z\subsetneq X$.
\end{theorem}

\subsection{Characterizations of Projective Spaces}
We recall some results about characterization of projective spaces among Fano manifolds. They are all in some sense related to the divisibility of the canonical divisor.
\begin{theorem}[\cite{Kobe}]\label{kob}
Let $X$ be a smooth Fano manifold of dimension $n$ and $H$ an ample divisor on $X$. If $-K_X\sim_{lin} lH$ with $l\geq n+1$, then $X\simeq \PP^n$, $l=n+1$, and $\OS_X(H)=\OS_{\PP^n}(1)$.
\end{theorem}
\begin{remark}\label{ntol}
Note that in the theorem it is enough to assume $-K_X\equiv_{num} lH$. Indeed, if $K_X+lH$ is numerically trivial, then we know that $\chi(K_X+lH)=\chi(\OS_X)=1$. By Kodaira vanishing, we have $h^i(X,K_X+lH)=0$ for $i>0$. Therefore $h^0(X,K_X+lH)=1$, which implies that $K_X+lH\sim_{lin} 0$.
\end{remark}

If we only consider K-semistable Fano manifolds of dimension $n$, we can characterize $\PP^n$ by the volume $\vol_X(-K_X)=(-K_X)^n$ instead, which is exactly the second part of Theorem \ref{fu}. Note that the condition $(-K_X)^n=(n+1)^n$ in Theorem \ref{fu} can be viewed as the divisibility of the $0$-cycle $(-K_X)^n$ by $(n+1)^n$, comparing to the divisibility of the divisor $-K_X$ by~$n+1$ in Theorem \ref{kob}. 

By considering the divisibility of cycles with intermediate codimension, we have the following immediate corollary of Theorem \ref{fu}:
\begin{corollary}\label{app}
Let $X$ be a K-semistable Fano manifold of dimension $n$. Suppose $k$ divides $n$. If $(-K_X)^k\equiv_{num}lZ$ for some rational number $l\geq (n+1)^k$ and $Z$ an integral $(n-k)$-cycle, then $X\simeq\PP^n$.
\end{corollary}
\begin{proof}
because $k$ divides $n$, we can write
$$
(-K_X)^n=\left((-K_X)^k)\right)^{\frac{n}{k}}=(lZ)^{\frac{n}{k}}\geq l^{\frac{n}{k}}\geq (n+1)^n.
$$
Since $X$ is K-semistable, we know from Theorem \ref{fu} that
$
(-K_X)^n= (n+1)^n
$
 and $X\simeq \PP^n$.
\end{proof}
For any dimension $n$, we at least know that $k$ divides $n$ when $k=1$ or $k=n$. These are the two cases discussed in Theorem \ref{kob} and Theorem \ref{fu} respectively. Along with some mild assumptions if necessary, we expect that the divisibility condition of the cycle $(-K_X)^k$ described in Corollary \ref{app} can characterize projective spaces among all K-semistable Fano manifolds. More precisely, we would like to ask the following question:
\begin{question}\label{con}
Let $X$ be a K-semistable Fano manifold of dimension $n$. If $(-K_X)^k\equiv_{num}lZ$ for some rational number $l\geq (n+1)^k$ and $Z$ an integral $(n-k)$-cycle, then is $X$ isomorphic to the projective space $\PP^n$?
\end{question}
Corollary \ref{app} answers Question \ref{con} when $k$ divides $n$. If the answer to Question \ref{con} is yes in general, then the three statements in Theorem \ref{eq} will be equivalent for all codimension $k$.
\section{Higher Codimensional Alpha Invariants}\label{min}
In this section, we will discuss some examples and properties of higher codimensional alpha invariants. We first recall the definition of higher codimensional alpha invariants.
\begin{definition}
Let $X$ be a $\Q$-Fano variety of dimension $n$. For $1\leq k\leq n$, denote by $\ci(L_1,\ldots,L_k)$ the complete intersection in $X$ cut out by effective Cartier divisors $L_1,\ldots,L_k$. Then we define the alpha invariant of codimension $k$ for $X$ to be
$$
\alpha^{(k)}(X):=\inf_{r}\left\{\lct\left(X,\frac{1}{r}Z\right)\bigg|Z=\ci(L_1,\ldots,L_k),~L_1,\ldots,L_k\in |-rK_X|\right\}.
$$
\end{definition}
\begin{remark}
Note that it follows immediately from the definition that $\alpha^{(k)}(X)$ can also be computed in terms of linear systems: 
$$
\alpha^{(k)}(X):=\inf_{r}\left\{\lct\left(X,\frac{1}{r}\HH\right)\bigg|\HH \subset |-rK_X|,~\codim \bs\HH\geq k\right\}.
$$
\end{remark}

A first basic property of alpha invariants is that all $\alpha^{(k)}(X)$'s form an increasing sequence in terms of the codimension $k$: $\alpha^{(1)}(X)\leq \alpha^{(2)}(X)\leq \cdots\leq \alpha^{(n)}(X)$. This follows immediately from the definition of alpha invariant and log canonical threshold.

The following proposition gives a lower bound for the top codimensional alpha invariant $\alpha^{(n)}(X)$.
\begin{proposition}[Proposition 2.4, \cite{zhuang2019birational}]\label{DF}
Let $X$ be a Fano manifold of dimension $n$. Then 
$
\alpha^{(n)}(X)\geq \frac{n}{\sqrt[n]{(-K_X)^n}}.
$
\end{proposition}

Combining Proposition \ref{DF} and Theorem \ref{fu}, we immediately have the following result which is a special case of Theorem \ref{eq} when $k=n$.
\begin{corollary}\label{casen}
Let $X$ be a K-semistable Fano manifold of dimension $n$. If $\alpha^{(n)}(X)\leq \frac{n}{n+1}$, then $X\cong \PP^n$. 
\end{corollary}

\begin{example}
Let $X$ be a del Pezzo surface of degree 1, then $\alpha^{(2)}(X)=2$. Indeed, by Proposition \ref{DF}, we have $\alpha^{(2)}(X)\geq 2$. By computing the log canonical threshold of a general pencil inside $|-K_X|$, we see that $\alpha^{(2)}(X)= 2$.
\end{example}

\begin{example}\label{expn}
Let $X$ be a Fano manifold with Fano index $l$. Assume that $-K_X\sim_{lin}lD$ such that the linear system $|D|$ is base point free. Then we know that $\alpha^{(k)}(X)\leq k/l$. Indeed we can take $k$ sufficiently general smooth elements $L_1,\ldots,L_k\in |D|$ such that $Z$ is the transversal intersection of $L_1,\ldots,L_k$. Then $\lct(X,Z)=k$, and therefore $\alpha^{(k)}(X)\leq \lct(X,lZ)=k/l$. This gives an upper bound of higher codimensional alpha invariants for smooth Fano hypersurfaces in the projective spaces. In particular, for the projective space $\PP^n$, we know that $\alpha^{(k)}(\PP^n)\leq k/(n+1)$. On the other hand, because $\PP^n$ is K-semistable, we know from Theorem \ref{ineq} that $\alpha^{(k)}(\PP^n)\geq k/(n+1)$. Consequently we have $\alpha^{(k)}(\PP^n)= k/(n+1)$, and it is realized by any linear subspaces of codimension $k$.
\end{example}
\begin{example}
Cheltsov computed in \cite{Cheltsov2008} that $\alpha^{(1)}(\PP^1\times \PP^1)=1/2$. However, even in simple examples, it seems hard to compute $\alpha^{(k)}$ when $k\geq 2$. For $\PP^1\times \PP^1$, we only know that $2/3<\alpha^{(2)}(\PP^1\times \PP^1)\leq 3/4$. Indeed, we first notice that $-K_{\PP^1\times \PP^1}$ is of type $(2,2)$. Let $L_1$ and $L_2$ be two lines of type $(1,0)$ and type $(0,1)$ respectively that are symmetric with respect to each other, and $\Delta$ be the diagonal in $\PP^1\times \PP^1$. Then $L_1+L_2$ and $\Delta$ are both linearly equivalent to $-\frac{1}{2}K_{\PP^1\times \PP^1}$. Let $Z$ be the complete intersection of  $L_1+L_2$ and $\Delta$, and we have $\lct(X,Z)=3/2$. Therefore $\alpha^{(2)}(\PP^1\times \PP^1)\leq \lct(X,2Z)=3/4$. We also know that $\PP^1\times \PP^1$ is K-semistable (for example refer to \cite{Park2017}). Then by Theorem \ref{ineq} and Theorem \ref{eq}, we have $\alpha^{(2)}(\PP^1\times \PP^1)>2/3$.
\end{example}

Recent results of \cite{stibitz2018k,zhuang2018birational} can be reinterpreted from the point of view of higher codimensional alpha invariants in the following way.
\begin{theorem}\label{fsz}
Let $X$ be a $\Q$-factorial $\Q$-Fano variety of Picard number 1 and dimension $n$. If 
$$
\alpha^{(2)}(X)> \frac{n-1}{n+1}
$$
and 
$$
\alpha^{(1)}(X)\geq \frac{\alpha^{(2)}(X)}{(n+1)\alpha^{(2)}(X)-n+1}~\left(resp.~\alpha^{(1)}(X)>\frac{\alpha^{(2)}(X)}{(n+1)\alpha^{(2)}(X)-n+1}\right),
$$
then $X$ is K-semistable $($resp. K-stable$)$.
\end{theorem}

\begin{remark}
The alpha invariant of codimension 2 is related to the notion of log maximal singularity. Recall that a $\Q$-factorial $\Q$-Fano variety $X$ of Picard number 1 has a log maximal singularity if there is a movable linear system $\HH$ on $X$ such that $\HH\equiv_{num}-rK_X$ and $(X,\frac{1}{r}\HH)$ is not log canonical. $X$ is called log maximal singularity free if $X$ does not have a log maximal singularity. Note that if a $\Q$-Cartier divisor $L\equiv_{num}-K_X$, then $L\sim_\Q-K_X$ by similar arguments in the smooth case as in Remark \ref{ntol}. Therefore we see that the linear system $\frac{1}{r}\HH$ we consider in the pair $(X,\frac{1}{r}\HH)$ is $\Q$-linear equivalent to $-K_X$. Then $X$ is log maximal singularity free if and only if $\alpha^{(2)}(X)\geq 1$. 
\end{remark}

Note that when $\alpha^{(2)}(X)\geq 1$, we have
$$
\frac{\alpha^{(2)}(X)}{(n+1)\alpha^{(2)}(X)-n+1}\leq \frac{1}{2}.
$$
Therefore Theorem \ref{fsz} reduces to the following theorem by Stibitz and Zhuang:
\begin{theorem}[\cite{stibitz2018k}]\label{sz1}
Let $X$ be a $\Q$-factorial $\Q$-Fano variety of Picard number 1. If $X$ is log maximal singularity free and $\alpha^{(1)}(X)\geq 1/2$ $(resp. >1/2)$, then $X$ is K-semistable $($resp. K-stable$)$.
\end{theorem}
Theorem \ref{fsz} is an immediate consequence of the following theorem of Zhuang, which provides a more precise result compared to Theorem \ref{sz1}.
\begin{theorem}[\cite{zhuang2018birational}]\label{sz2}
Let $X$ be a $\Q$-factorial $\Q$-Fano variety of Picard number 1 and dimension $n$. If for every effective divisor $D\sim_\Q -K_X$ and every movable linear system $M\sim_\Q -K_X$, we have that the pair $(X,\frac{1}{n+1}D+\frac{n-1}{n+1}M)$ is log canonical $($resp. klt$)$, then $X$ is K-semistable $($resp. K-stable$)$.
\end{theorem}
Indeed, pick any effective divisor $D\sim_\Q -K_X$ and movable linear system $M\sim_\Q -K_X$. Then we know that the pairs $(X,\alpha^{(1)}(X)D)$ and $(X,\alpha^{(2)}(X)M)$ are log canonical. We can write $\frac{1}{n+1}D+\frac{n-1}{n+1}M$ as a linear combination of $\alpha^{(1)}(X)D$ and $\alpha^{(2)}(X)M$ as follows:
$$
\frac{1}{n+1}D+\frac{n-1}{n+1}M=\frac{1}{(n+1)\alpha^{(1)}(X)}\alpha^{(1)}(X)D+\frac{n-1}{(n+1)\alpha^{(2)}(X)}\alpha^{(2)}(X)M
$$
Note that the conditions
$$
\alpha^{(2)}(X)> \frac{n-1}{n+1}
$$
and
$$
\alpha^{(1)}(X)\geq \frac{\alpha^{(2)}(X)}{(n+1)\alpha^{(2)}(X)-n+1}
$$
are equivalent to
\begin{equation}\label{com}
\frac{1}{(n+1)\alpha^{(1)}(X)}+\frac{n-1}{(n+1)\alpha^{(2)}(X)}\leq 1.
\end{equation}
Then the pair $(X,\frac{1}{n+1}D+\frac{n-1}{n+1}M)$ is also log canonical and by Theorem \ref{sz2}, $X$ is K-semistable . Note that when 
$$
\alpha^{(1)}(X)>\frac{\alpha^{(2)}(X)}{(n+1)\alpha^{(2)}(X)-n+1}
$$
we will have strict inequality in \eqref{com} instead. Then $(X,\frac{1}{n+1}D+\frac{n-1}{n+1}M)$ is klt and $X$ is K-stable.
\begin{remark}
The above theorems are of course related to the well-known result of Tian in \cite{tian_alpha}, which was later improved by Fujita in \cite{fujita_2017} stating that a Fano manifold $X$ of dimension $n$ is K-stable if $\alpha^{(1)}(X)\geq n/(n+1)$. Theorem \ref{fsz} does not require smoothness, and the lower bound of $\alpha^{(1)}(X)$ is smaller at the cost of an additional assumption on $\alpha^{(2)}(X)$ and the Picard number.
\end{remark}
\begin{remark}
In \cite{Tianhigher2,macbeth2014k}, they give a definition of higher alpha invariants. Although in general these higher alpha invariants are different from $\alpha^{(k)}(X)$'s defined in this paper, they come up with similar results as Theorem \ref{fsz}.
\end{remark}

\section{Proof of Theorem \ref{ineq}}\label{pre}
In this section, we prove Theorem \ref{ineq} which gives a lower bound of higher codimensional alpha invariants for K-semistable $\Q$-Fano varieties. We first state 2 lemmas that will be used in later computation.
\begin{lemma}\label{vog}
Let $X$ be a normal projective variety of dimension $n$ with klt singularities, and $L$ an ample divisor on $X$. Let $Z$ be a complete intersection of $X$ cut out by $k$ elements $L_1,\ldots,L_k$ in the linear system $|L|$ with ideal sheaf $I_Z\subset \OS_X$. Let $\sigma:Y\to X$ be any proper birational morphism that factors through the blow-up of $X$ along $Z$, and write $\sigma^{-1}I_Z\cdot \OS_Y=\OS_Y(-F)$. Let $\HH$ be the sub linear system of $|L|$ with base ideal $I_Z$, and $\sigma^\ast\HH=|M|+F$ be the decomposition into the moving part and the fixed part of the linear system $\sigma^\ast\HH$. Then we have
$$
\sigma_\ast (M^{i-1}\cdot F)=\left\{
\begin{array}{cl}
[Z]&,~i=k;\\
0&,~i\neq k,
\end{array}\right.
$$
and in particular,
$$
\sigma^\ast L^{n-i}\cdot(M^{i-1}\cdot F)=\left\{
\begin{array}{cl}
\deg_L(Z)=L^n&,~i=k;\\
0&,~i\neq k.
\end{array}\right.
$$
\end{lemma}
\begin{proof}
First note that the linear system $|M|$ corresponds to the global sections of $\OS_Y(\sigma^\ast L-F)$ and is base point free. Write $\deg_L(Z)$ as
$$
\deg_L(Z)=L^n=\sum_{i=1}^n \sigma^\ast{L}^{n-i}\cdot(M^{i-1}F)+M^n.
$$
For any $i$, we know that $\sigma^\ast{L}^{n-i}\cdot(M^{i-1}F)\geq 0$ and $M^n\geq 0$. Therefore, we only need to show $[Z]=\sigma_\ast(M^{k-1}F)$.

We write $[Z]=\sum a_i[Z_i]$ with $a_i=l(\OS_{Z,Z_i})$ and $Z_i$'s to be irreducible components of $Z$. Localizing at one $Z_i$ and let $U=\spec \OS_{X,Z_i}$. Then we know that $\sigma^\ast{L}$ is linear equivalent to zero over $U$. Then over $U$, we have that $M\sim_{lin}-F$ , and $\sigma_\ast (M^{k-1}F)=(-1)^{k-1}\sigma_\ast (F^k)=e(I_Z\cdot \OS_{X,Z_i})[Z_i]$ by formula \eqref{ram}. Note that $X$ is klt. Then in particular $X$ is Cohen-Macaulay. Therefore we can apply Proposition \ref{cim} to the complete intersection $Z$ so that we know $e(I_Z\cdot \OS_{X,Z_i})$ can be computed by the length of $\OS_{{X,Z_i}}/(I_Z\cdot \OS_{{X,Z_i}})=\OS_{Z,Z_i}$. Therefore we have $e(I_Z\cdot \OS_{X,Z_i})=a_i$. Now by localizing at all $Z_i$'s, we see that $\sigma_\ast (M^{k-1}F)=\sum a_i[Z_i]=[Z]$.
\end{proof}
\begin{proof}[Proof of Theorem \ref{ineq}]
Let $Z=\ci(L_1,\ldots,L_k)$ where $L_1,\ldots,L_k\in |-rK_X|$. Assume $\sigma:Y\to X$ is a proper birational morphism that factors through the blow-up $X$ along $Z$. Write $\sigma^{-1}I_Z\cdot \OS_Y=\OS_Y(-F)$ for some Cartier divisor $F$ on $Y$. Let $\epsilon(Z,-K_X)$ be the Seshadri constant of $Z$ with respect to $-K_X$. By the construction of $Z$ we know that $I_Z\cdot\OS_X(-rK_X)$ is globally generated. Hence $\epsilon(Z,-K_X)\geq 1/r$. In particular for $x< 1/r$, we know that $\sigma^\ast(-K_X)-xF$ is the pullback of an ample line bundle on the blow-up of $X$ along $Z$, and hence nef and big. Then for $0<x<1/r$, we have
$$
\vol_Y(\sigma^\ast(-K_X)-xF)=(\sigma^\ast(-K_X)-xF)^n.
$$
Applying Theorem \ref{betaZ} to $Z$, we have
$$
\beta(X,Z):=\lct(X,Z)\vol_X(-K_X)-\int_0^\infty \vol_Y(\sigma^\ast(-K_X)-xF)dx\geq 0.
$$
Consequently, we know that
\begin{align}\label{hold}
\nonumber \lct(X,Z)&\geq \frac{1}{\vol_X(-K_X)}\int_0^\infty \vol_Y(\sigma^\ast(-K_X)-xF)dx\\
\nonumber &\geq \frac{1}{\vol_X(-K_X)}\int_0^\frac{1}{r} \vol_Y(\sigma^\ast(-K_X)-xF)dx\\
&=\frac{1}{r(-rK_X)^n}\int_0^1 (\sigma^\ast(-rK_X)-xF)^n dx.
\end{align}
Integrating by parts $k$ times, we get that
$$
\int_0^1 \left(\sum_{i=0}^{n-k} \bi{n-1-i}{k-1}(1-x)^{n-k-i}x^{k}\right)dx=1-\frac{k}{n+1}.
$$
Therefore, using formula \eqref{bypar} in Lemma \ref{poly} below, we have
$$
\int_0^1 (\sigma^\ast(-rK_X)-xF)^n dx=(-rK_X)^n\cdot\frac{k}{n+1}.
$$
Then by \eqref{hold} we know that $\displaystyle{\lct(X,\frac{1}{r}Z)=r\lct(X,Z)\geq \frac{k}{n+1}}$ for any choice of $Z$, which implies that $\alpha^{(k)}(X)\geq \displaystyle{\frac{k}{n+1}}$.
\end{proof}
We close this section with the following lemma which we have used in the proof above. The lemma also includes formula \eqref{bin} which will be used in the next section for the proof of Theorem \ref{eq}. 
\begin{lemma}\label{poly}
Under the setting of Lemma \ref{vog},  we have the following two formulas for the polynomial $(\sigma^\ast L-xF)^n$:
\begin{equation}\label{bypar}
(\sigma^\ast L-xF)^n=L^n\left(1-\sum_{i=0}^{n-k} \bi{n-1-i}{k-1}(1-x)^{n-k-i}x^{k}\right)
\end{equation}

\begin{equation}\label{bin}
(\sigma^\ast L-xF)^n=L^n\left(1+\sum_{i=k}^n (-1)^{i+k-1}\bi{n}{i} \bi{i-1}{k-1}x^i\right).
\end{equation}

\end{lemma}
\begin{proof}
First of all, by Lemma \ref{vog}, we have
\begin{align*}
L^n-(\sigma^\ast L-xF)^n&=xF\cdot\left(\sum_{i=0}^{n-1} \sigma^\ast L^i\bigg((1-x)\sigma^\ast L+xM\bigg)^{n-1-i}\right)\\
&=L^n\cdot\left(\sum_{i=0}^{n-k} \bi{n-1-i}{k-1}(1-x)^{n-k-i}x^{k}\right).
\end{align*}
This gives us formula \eqref{bypar}.

On the other hand, expand the intersection number $(\sigma^\ast L-xF)^n$, we have
$$
(\sigma^\ast L-xF)^n=L^n+\sum_{i=1}^n(-1)^i\bi{n}{i}\left(\sigma^\ast L^{n-i}\cdot F^i\right)x^i.
$$
Note that by Lemma \ref{vog}, for $i\geq k$,
$$
\sigma^\ast L^{n-i}\cdot F^i=F\cdot \sigma^\ast L^{n-i}(\sigma^\ast L-M)^{i-1}=(-1)^{k-1}\bi{i-1}{k-1}L^{n},
$$
and $\sigma^\ast L^{n-i}\cdot F^i=0$ if $i<k$. Therefore, we get formula \eqref{bin}.
\end{proof}
\section{Proof of Theorem \ref{eq}}
We first prove the following proposition, which gives the implication $(2)\Rightarrow (3)$ in Theorem \ref{eq}.
\begin{proposition}\label{pop}
Let $X$ be a smooth K-semistable Fano variety of dimension $n$. Let $Z$ be a complete intersection of $L_1,\ldots,L_k\in |-rK_X|$ such that $\lct(X,\frac{1}{r}Z)=\frac{k}{(n+1)}$.
Then $Z$ is irreducible. Furthermore, we have $(-K_X)^k\sim_{rat} lZ'$ for some rational number $l\geq (n+1)^k$, where $Z'$ is the integral $(n-k)$-cycle corresponding to the support of $Z$ with the reduced scheme structure.
\end{proposition}
\begin{proof}

Let $\sigma:Y\to X$ be a log resolution of $(X,Z)$, with $\sigma^{-1}I_Z\cdot \OS_Y=\OS_Y(-F)$ for some Cartier divisor $F$ on $Y$. Write $F=\sum a_iE_i$, where $a_i=\ord_{E_i}(Z)$. Because $X$ is smooth, we have 
$$
\lct(X,Z)=\min_{i}\frac{A_X(E_i)}{a_i}.
$$
Let $E$ be a divisor that computes $\lct(X,Z)$ and $a=\ord_E(Z)$. We want to show first that the center of $E$ on $X$ has the same dimension as $Z$. By Theorem \ref{beta}, we know that $\beta(E)\geq 0$, so 
$$
\frac{A_X(E)}{a}\geq \frac{1}{a(-K_X)^n}\int_0^\infty \vol_Y(\sigma^\ast(-K_X)-xE)dx.
$$
Then because all the equality holds in \eqref{hold}, we have
$$
\frac{1}{(-K_X)^n}\int_0^\infty \vol_Y(\sigma^\ast(-K_X)-xF)dx=\lct(X,Z)=\frac{A_X(E)}{a}.
$$
Consequently,
$$
\frac{1}{(-K_X)^n}\int_0^\infty \vol_Y(\sigma^\ast(-K_X)-xF)dx\geq \frac{1}{a(-K_X)^n}\int_0^\infty \vol_Y(\sigma^\ast(-K_X)-xE)dx.
$$
By a change of variable for the integral on the RHS of the above inequality, we get
$$
\frac{1}{(-K_X)^n}\int_0^\infty \vol_Y(\sigma^\ast(-K_X)-xF)dx\geq\frac{1}{(-K_X)^n}\int_0^\infty \vol_Y(\sigma^\ast(-K_X)-xaE)dx.
$$
Also, because $F\geq aE$, for all $x$ we have
$$
\vol_Y(\sigma^\ast(-K_X)-xF)\leq \vol_Y(\sigma^\ast(-K_X)-xaE).
$$
Then there is an equality of volumes:
\begin{equation}\label{vo}
\vol_Y(\sigma^\ast(-K_X)-xF)=\vol_Y(\sigma^\ast(-K_X)-xaE),
\end{equation}
for all $x$. 

Now let $\aid_m=\sigma_\ast\OS_Y(-mE)$ for any integer $m$. Then $\aid_m$ defines a subscheme of $X$ the support of which is $\sigma(E)$. 
Combining the conclusions from \cite[Theorem 1.3, Theorem 1.4, Proposition 1.5]{blum2016divisors}, we see that the graded sequence of ideals $\aid_\bullet$ is finitely generated. In addition, suppose $\aid_\bullet$ is generated in degree up to $m$. Then we have that 
$$
\pi:W:=Bl_{\aid_m} X\to X
$$
is the blow-up of $X$ along $\aid_m$, and $\pi^{-1}\aid_m\cdot\OS_W=\OS_W(-mE_W)$, where $E_W$ is the prime exceptional divisor on $W$ that induces the same divisorial valuation as $E$ on $Y$.

By replacing $Y$ with a common log resolution of $Y$ and $W=Bl_{\aid_m} X$, we may assume that $\sigma:Y\to X$ factors through the blow-up $\pi:W \to X$, and $\sigma^{-1}\aid_{am}\cdot\OS_Y=\OS(-mD)$ for some divisor $D$ on $Y$. Note that $D$ is  the pullback of $aE_W$ from $W$ to $Y$. Also because $aE\leq F$, we have 
$$
\sigma_\ast\OS_Y(-mF)\subset \sigma_\ast\OS_Y(-maE)=\aid_{am},
$$
Then consider the inverse image sheaves of corresponding ideals under the map $\sigma$, we have that $mD\leq mF$. Therefore the divisor $D$ satisfies the relation
$$
aE\leq D\leq F.
$$
Now that we have \eqref{vo}, we get the following equality:
\begin{equation*}
\vol_Y(\sigma^\ast(-K_X)-xF)=\vol_Y(\sigma^\ast(-K_X)-xD)=\vol_Y(\sigma^\ast(-K_X)-xaE).
\end{equation*}
Since $\sigma^\ast(-K_X)-xD$ is the pullback of $\pi^\ast(-K_X)-xaE_W$ from $W$ to $Y$,  we have the equality of the volumes:
\begin{equation}\label{vol}
\vol_Y(\sigma^\ast(-K_X)-xF)=\vol_W(\pi^\ast(-K_X)-xaE_W).
\end{equation}

Assume $\sigma(E)=\pi(E_W)$ is of codimension $s$ in $X$. Now on the left hand side of \eqref{vol}, for $x< 1/r$, by a change of variable, we know from formula \eqref{bin} that
\begin{equation}\label{lhs}
\vol_Y(\sigma^\ast(-K_X)-xF)=(-K_X)^n\cdot\left(1+\sum_{i=k}^n (-1)^{i+k-1}\bi{n}{i} \bi{i-1}{k-1}(rx)^i\right).
\end{equation}
On the right hand side of \eqref{vol}, for sufficiently small $x$, the divisor $\pi^\ast(-K_X)-xaE_W$ is ample on the blow-up $W$. Therefore we have
$$
\vol_W(\pi^\ast(-K_X)-xaE_W)=(\pi^\ast(-K_X)-xaE_W)^n.
$$
Expanding the intersection number, we get a polynomial in terms of $x$. Note that the dimension of $\pi(E_W)$ is $n-s$, so we have $(\pi^\ast(-K_X))^iE_W^{n-i}=0$ when $n-s<i<n$. Therefore
\begin{equation}\label{rhs}
(\pi^\ast(-K_X)-xaE_W)^n=(-K_X)^n+(-1)^s\bi{n}{s}(\pi^\ast(-K_X))^{n-s}a^sE_W^{s}x^s+O(x^{s+1}).
\end{equation}
Compare the above two polynomials on the right hand side of \eqref{lhs} and \eqref{rhs}. We see that $k=s$, so $\dim \sigma(E)=n-k=\dim Z$. Therefore we know that the center of $E$ on $X$ is an irreducible component of $Z$.

Next, we want to show that $Z$ is irreducible. Using formula \eqref{bin} for $\vol(\sigma^\ast (-rK_X)-xF)$, we see that the coefficient of $x^k$ in $\vol(\sigma^\ast (-rK_X)-xF)$ is $-(-rK_X)^n\bi{n}{k}$. If we compare it with the coefficient of $x^k$ in the expansion  of the polynomial $\vol(\pi^\ast (-rK_X)-xaE_W)=(\pi^\ast (-rK_X)-xaE_W)^n$, we get that
\begin{equation}\label{ded}
(-rK_X)^n=a^k(-1)^{k-1}\pi^\ast (-rK_X)^{n-k}E_W^k.
\end{equation}
Now suppose $Z$ is not irreducible. Write $[Z]=a'Z'+\sum_i a_i[Z_i]$, where $Z'$ is the $(n-k)$-cycle defined by the reduced irreducible component corresponding to the center of $E$. Note that 
$$
\aid_m^a=\aid_{am}\supset \sigma_\ast \OS_Y(-mF)=\overline{I^m_Z},
$$
where $\overline{I^m_Z}$ is the integral closure of the ideal $I^m_Z$. Since $\overline{I^m_Z}$ has the same multiplicity as $I^m_Z$, we know that
$$
a^ke(\aid_m\cdot \OS_{X,Z'})\leq e(\overline{I^m_Z}\cdot \OS_{X,Z'})= m^ke(I_Z\cdot \OS_{X,Z'}),
$$
By the computation in the proof of Lemma \ref{vog}, we know that 
$$
e(I_Z\cdot \OS_{X,Z'})=l(\OS_{Z,Z'})=a'.
$$
Therefore, we know that 
$$
a'\geq \frac{a^k}{m^k}e(\aid_m\cdot \OS_{X,Z'}).
$$
By \eqref{ram}, we have
$$(-1)^{k-1}\pi_\ast E_W^k=\frac{1}{m^k}e(\aid_m\cdot \OS_{X,Z'})Z'
$$ 
for sufficiently divisible $m$. Together with \eqref{ded}, we have
$$
(-rK_X)^{n-k}\cdot\left(\left(a'-\frac{a^k}{m^k}e(\aid_m\cdot \OS_{X,Z'})\right)Z'+\sum_i a_i[Z_i]\right)=0.
$$
Since the cycle we intersect with $(-rK_X)^{n-k}$ in the above equation is effective, we know that $Z$ is irreducible supporting on $Z'$ and
\begin{equation*}
(-rK_X)^n= \deg_{(-rK_X)} Z' \frac{a^k}{m^k}e(\aid_m\cdot \OS_{X,Z'}).
\end{equation*}
Note that 
$
\frac{A_X(E)}{a}=\lct(X,Z)=\frac{k}{(n+1)r},
$ 
so we have
$
a=\frac{A_X(E)(n+1)r}{k}.
$
Consequently,
\begin{equation}\label{im}
(-K_X)^n= \deg_{(-K_X)}Z' \left(\frac{n+1}{k}\right)^k\frac{A_X(E)^ke(\aid_m\cdot \OS_{X,Z'})}{m^k}.
\end{equation}
The log discrepancy of $E$ doesn't change after we localize at $Z'$ because $Z'$ is the center of $E$. Working on $\spec \OS_{X,Z'}$, we next want to show that 
$$
\frac{A_X(E)^ke(\aid_m\cdot \OS_{X,Z'})}{m^k}\geq k^k.
$$
Note that by the definition of log canonical threshold and $\aid_m=\sigma_\ast \OS_Y(-mE)$, we have
$$
\lct(\aid_m\cdot \OS_{X,Z'})\leq \frac{A_X(E)}{\ord_E(\aid_m\cdot \OS_{X,Z'})}=\frac{A_X(E)}{m}.
$$
Since $\spec \OS_{X,Z'}$ is smooth, we know that
$$
\frac{A_X(E)^ke(\aid_m\cdot \OS_{X,Z'})}{m^k}\geq \lct(\aid_m\cdot \OS_{X,Z'})^ke(\aid_m\cdot \OS_{X,Z'})\geq k^k,
$$
where the last inequality follows from \cite{DFEM}. Therefore from \eqref{im} we have the following inequality:
$$
(-K_X)^n\geq (n+1)^k(-K_X)^{n-k}Z',
$$
or equivalently, 
$$
(-K_X)^{n-k}[(-K_X)^k-(n+1)^kZ']\geq 0.
$$
Since $(-K_X)^k\sim_{rat} \frac{1}{r^k}Z$, we may assume $(-K_X)^k\sim_{rat} lZ'$ for some rational number $l$. Then by the ampleness of $-K_X$, we know that $l\geq (n+1)^k$. This finishes the proof of Proposition~\ref{pop}.
\end{proof}
\begin{proof}[Proof of Theorem \ref{eq}]
We have already seen in Example \ref{expn} that $\alpha^{(k)}(\PP^n)=k/(n+1)$ and it is realized by linear subspaces of codimension $k$, which gives $(1)\Rightarrow (2)$. We also have $(2)\Rightarrow (3)$ by Proposition \ref{pop}. 
Finally by Corollary~\ref{app}, we know that when $k$ divides $n$, we have $(3)\Rightarrow (1)$ in Theorem \ref{eq}.
\end{proof}

\bibliographystyle{alpha}
\bibliography{ref}

\begin{thebibliography}{dFEM04}

\bibitem[Bir16]{birkar2016singularities}
Caucher Birkar.
\newblock Singularities of linear systems and boundedness of {F}ano varieties.
\newblock {\em arXiv preprint arXiv:1609.05543v1}, 2016.

\bibitem[Blu16]{blum2016divisors}
Harold Blum.
\newblock On divisors computing {MLD}'s and {LCT}'s.
\newblock {\em arXiv preprint arXiv:1605.09662v3}, 2016.

\bibitem[CDS15a]{D1}
Xiuxiong Chen, Simon Donaldson, and Song Sun.
\newblock K\"ahler-{E}instein metrics on {F}ano manifolds. {I}: {A}pproximation
  of metrics with cone singularities.
\newblock {\em J. Amer. Math. Soc.}, 28(1):183--197, 2015.

\bibitem[CDS15b]{D2}
Xiuxiong Chen, Simon Donaldson, and Song Sun.
\newblock K\"ahler-{E}instein metrics on {F}ano manifolds. {II}: {L}imits with
  cone angle less than {$2\pi$}.
\newblock {\em J. Amer. Math. Soc.}, 28(1):199--234, 2015.

\bibitem[CDS15c]{D3}
Xiuxiong Chen, Simon Donaldson, and Song Sun.
\newblock K\"ahler-{E}instein metrics on {F}ano manifolds. {III}: {L}imits as
  cone angle approaches {$2\pi$} and completion of the main proof.
\newblock {\em J. Amer. Math. Soc.}, 28(1):235--278, 2015.

\bibitem[Che08]{Cheltsov2008}
Ivan Cheltsov.
\newblock Log canonical thresholds of {D}el {P}ezzo surfaces.
\newblock {\em Geometric and Functional Analysis}, 18(4):1118--1144, Dec 2008.

\bibitem[dFEM04]{DFEM}
Tommaso de~Fernex, Lawrence Ein, and Mircea Musta\c{t}\u{a}.
\newblock Multiplicities and log canonical threshold.
\newblock {\em J. Algebraic Geom.}, 13(3):603--615, 2004.

\bibitem[FO16]{fujita2016k}
Kento Fujita and Yuji Odaka.
\newblock On the {K}-stability of {F}ano varieties and anticanonical divisors.
\newblock {\em arXiv preprint arXiv:1602.01305v2}, To appear in Tohoku Math.
  J., 2016.

\bibitem[Fuj16]{fujita2016valuative}
Kento Fujita.
\newblock A valuative criterion for uniform {K}-stability of {${\mathbb
  Q}$}-{F}ano varieties.
\newblock {\em J. Reine Angew. Math.}, Published online, 2016.

\bibitem[Fuj17]{fujita_2017}
Kento Fujita.
\newblock K-stability of {F}ano manifolds with not small alpha invariants.
\newblock {\em Journal of the Institute of Mathematics of Jussieu}, page
  1–12, 2017.

\bibitem[Fuj18]{fujita2015optimal}
Kento Fujita.
\newblock Optimal bounds for the volumes of {K}{\"a}hler-{E}instein {F}ano
  manifolds.
\newblock {\em Amer. J. Math.}, 140(2):391--414, 2018.

\bibitem[Ful12]{fulton2012intersection}
W.~Fulton.
\newblock {\em Intersection Theory}.
\newblock Ergebnisse der Mathematik und ihrer Grenzgebiete. Springer New York,
  2012.

\bibitem[Jia17]{sma}
Chen Jiang.
\newblock K-semistable {F}ano manifolds with the smallest alpha invariant.
\newblock {\em Internat. J. Math.}, 28(6):1750044, 9, 2017.

\bibitem[KM08]{kollar2008birational}
J.~Koll{\'a}r and S.~Mori.
\newblock {\em Birational Geometry of Algebraic Varieties}.
\newblock Cambridge Tracts in Mathematics. Cambridge University Press, 2008.

\bibitem[KO73]{Kobe}
Shoshichi Kobayashi and Takushiro Ochiai.
\newblock Characterizations of complex projective spaces and hyperquadrics.
\newblock {\em J. Math. Kyoto Univ.}, 13:31--47, 1973.

\bibitem[Laz04]{lazarsfeld2004positivity2}
R.K. Lazarsfeld.
\newblock {\em Positivity in Algebraic Geometry II: Positivity for Vector
  Bundles, and Multiplier Ideals}.
\newblock Ergebnisse der Mathematik und ihrer Grenzgebiete. 3. Folge / A Series
  of Modern Surveys in Mathematics. Springer Berlin Heidelberg, 2004.

\bibitem[Li17]{li2015k}
Chi Li.
\newblock K-semistability is equivariant volume minimization.
\newblock {\em Duke Math. J.}, 166(16):3147--3218, 2017.

\bibitem[Liu16]{liu2016volume}
Yuchen Liu.
\newblock The volume of singular {K}\"{a}hler-{E}instein {F}ano varieties.
\newblock {\em arXiv preprint arXiv:1605.01034v4}, To appear in Compositio
  Mathematica, 2016.

\bibitem[LZ17]{Liu2017}
Yuchen Liu and Ziquan Zhuang.
\newblock Characterization of projective spaces by seshadri constants.
\newblock {\em Mathematische Zeitschrift}, Published online, 2017.

\bibitem[Mac14]{macbeth2014k}
Heather Macbeth.
\newblock K\"ahler-{E}instein metrics and higher alpha-invariants.
\newblock {\em arXiv preprint arXiv:1411.7366}, 2014.

\bibitem[PW17]{Park2017}
Jihun Park and Joonyeong Won.
\newblock K-stability of smooth del pezzo surfaces.
\newblock {\em Mathematische Annalen}, Published online, 2017.

\bibitem[SZ18]{stibitz2018k}
Charlie Stibitz and Ziquan Zhuang.
\newblock K-stability of birationally superrigid {F}ano varieties.
\newblock {\em arXiv preprint arXiv:1802.08381}, 2018.

\bibitem[Tia87]{tian_alpha}
Gang Tian.
\newblock On {K}\"ahler-{E}instein metrics on certain {K}\"ahler manifolds with
  {$C_1(M)>0$}.
\newblock {\em Invent. Math.}, 89(2):225--246, 1987.

\bibitem[Tia91]{Tianhigher2}
Gang Tian.
\newblock On one of {C}alabi's problems.
\newblock In {\em Several complex variables and complex geometry, {P}art 2
  ({S}anta {C}ruz, {CA}, 1989)}, volume~52 of {\em Proc. Sympos. Pure Math.},
  pages 543--556. Amer. Math. Soc., Providence, RI, 1991.

\bibitem[Tia15]{tian}
Gang Tian.
\newblock K-stability and {K}\"ahler-{E}instein metrics.
\newblock {\em Comm. Pure Appl. Math.}, 68(7):1085--1156, 2015.

\bibitem[Zhu18]{zhuang2018birational}
Ziquan Zhuang.
\newblock Birational superrigidity and {K}-stability of {F}ano complete
  intersections of index one (with an appendix written jointly with {C}harlie
  {S}tibitz).
\newblock {\em arXiv preprint arXiv:1802.08389}, 2018.

\bibitem[Zhu19]{zhuang2019birational}
Ziquan Zhuang.
\newblock Birational superrigidity is not a locally closed property.
\newblock {\em arXiv preprint arXiv:1901.00078}, 2019.

\end{thebibliography}
\end{document}